\documentclass[12pt,letterpaper,reqno]{amsart}
\usepackage[top=1in, bottom=1in, left=1in, right=1in]{geometry}
\usepackage{enumerate}
\usepackage{amsmath,amsthm,amssymb,amsfonts,latexsym}
\usepackage{mathtools}
\usepackage{hyperref}
\usepackage[gen]{eurosym}
\usepackage[german,english]{babel}
\usepackage{dsfont}
\usepackage{tikz}
\usetikzlibrary{calc}
\usetikzlibrary{positioning}
\allowdisplaybreaks
\usepackage{comment}
\usepackage{float}
\usepackage{latexsym}

\usepackage{tikz}					% Grafiken zeichnen

\newtheorem*{theorem*}{Theorem}
\newtheorem*{prop*}{Proposition}

\DeclareMathOperator{\supp}{supp}

\newtheorem{theorem}{Theorem}%[section]
\newtheorem{lemma}[theorem]{Lemma}

\newtheorem{remark}[theorem]{Remark}

\newtheorem{example}[theorem]{Example}

\newcommand{\be}{\begin{equation}}
\newcommand{\ee}{\end{equation}}
\newcommand{\bea}{\begin{eqnarray*}}
	\newcommand{\eea}{\end{eqnarray*}}
\newcommand{\beq}{\begin{eqnarray}}
\newcommand{\eeq}{\end{eqnarray}}

\newtheorem{cor}[theorem]{Corollary}

\usepackage{todonotes}

\title[Ambarzumian-type theorems for Hermitian matrices with applications]{Ambarzumian-type theorems for Hermitian matrices with applications}



\author[M.~Hofmann]{Matthias Hofmann}
\author[J.~Kerner]{Joachim Kerner}

\address{Matthias Hofmann, Lehrgebiet Numerische Mathematik, Fakult\"at Mathematik und Informatik, Fern\-Universit\"at in Hagen, 58084 Hagen, Germany}
\email{matthias.hofmann@fernuni-hagen.de}

\address{Joachim Kerner, Lehrgebiet Analysis, Fakult\"at Mathematik und Informatik, Fern\-Universit\"at in Hagen, 58084 Hagen, Germany}
\email{joachim.kerner@fernuni-hagen.de}

\date{\today}

\begin{document}
	
	\begin{abstract} 
        A foundational result in inverse spectral theory due to Ambarzumian (1929) states that the Neumann Laplacian on an interval is not isospectral to the Neumann Laplacian with an additional non-zero potential. In this note, our aim is to investigate Ambarzumian-type theorems for Hermitian matrices, including well-known matrices such as the discrete Laplacian and the adjacency matrix on finite graphs. We also provide some genericity results which show that spectral stability in the sense of Ambarzumian is typically absent.
    \end{abstract}
	
\maketitle

\section{Introduction}
In inverse spectral theory one tries to recover as much information
as possible about a system using only spectral properties~\cite{Borg,LevinsomInverse,HochSL}. Highly stimulating in this context was the famous question \textit{Can one hear the shape of a drum?} asked by M.~Kac~\cite{MKac}. Later, U.~Smilansky and his collaborator formulated the related question \textit{Can one hear a matrix?} in \cite{SMMatrix}. It is known that the general answer to this question is \textit{No!} (see the first line in their paper), and by studying Ambarzumian-type results for Hermitian matrices, we aim at contributing to an understanding thereof.  

The starting point of our investigation is the paper \cite{Ambarz} written by Ambarzumian in the 1920s. In his well-known paper, Ambarzumian showed that the Neumann Laplacian on an interval is not isospectral to the corresponding Schrödinger operator on the same interval including an additional non-zero potential; in this sense, one may indeed hear the presence of an external potential. On the other hand, and this is important to keep in mind, it is known that such Ambarzumian-type theorems do not always hold~\cite{PKUnderstanding,BKAmba}: for example, one can explicitly construct a Schrödinger operator on an interval including a potential and subject to Robin boundary conditions that is isospectral to the Neumann Laplacian. Consequently, one cannot hear the presence of the external potential in this case. 

Since the foundational paper of Ambarzumian, his results have been generalized in many different directions~\cite{HarrellAmba}, including quantum graphs~\cite{DaviesAmba,BKS,PivoAmba,KissAmba,BKAmba,BrolinAmba} and some discrete graphs~\cite{AmbaDisc}. In this note, we discuss Ambarzumian-type theorems in greater generality for certain classes of Hermitian matrices. Generalizing ideas from \cite{DaviesAmba} (related results can be found in \cite{HarrellAmba} who traces them back to \cite{LevGasAmba}), we derive a Davies-type theorem for Hermitian matrices (Theorem~\ref{MainResultAmbaI}) from which we obtain a certain Ambarzumian-type theorem as a corollary (Theorem~\ref{MainResultAmbaII}); in particular, this result covers the discrete Laplacian on finite graphs (Theorem~\ref{CorollaryLaplacian}). Furthermore, using some rather elementary relations, we also prove an Ambarzumian-type theorem for the adjacency matrix of a finite discrete graph (Theorem~\ref{MainResult4}); by doing this, we recover results from~\cite{AmbaDisc} and generalize them to arbitrary graphs. In Section~\ref{SectionMainResultsII} we then prove an abstract Ambarzumian-type theorem that generalizes some previous results (Theorem~\ref{TheoremKyFan}) and  -- most importantly -- we derive some genericity results which show that Ambarzumian-type theorems typically cannot be expected to hold (Theorem~\ref{MainTheoremGeneric} and Corollary~\ref{CorollaryGeneric}). 

Finally, for some results regarding Ambarzumian-type theorems on infinite discrete graphs let us refer to \cite{BKR} and \cite{SaburovaInfinite}. Whereas the graphs in \cite{SaburovaInfinite} are periodic, the graphs in \cite{BKR} are typically not but the authors assume that the discrete Laplacian has purely discrete spectrum.

\section{Some motivating discussion}
As a first step, we formulate the following question:
For $n \geq 2$, let $H_0 \in \mathbb{C}^{n \times n}$ be a Hermitian matrix and $V=\mathrm{diag}(v_j) \in \mathbb{R}^{n \times n}$ a diagonal matrix. Is it true that $H:=H_0+V$ and $H_0$ are isospectral if and only if $V=0$? Note here that we call two Hermitian matrices $A,B \in \mathbb{C}^{n \times n}$ isospectral if and only if they have the same eigenvalues (counted with multiplicities). At this point, let us also refer to \cite{Friedland} (in particular, Theorem~2.2) and \cite{DowningHouseholder,Hadeler,Alexander} where related questions have been studied. 

A simple calculation shows that, in this generality, the statement is wrong. To see this, consider two-dimensional real matrices
\begin{equation*}H_0=\begin{pmatrix} a & b \\ b & c
\end{pmatrix}\ , \quad  
V=\begin{pmatrix} \alpha & 0 \\ 0 & \beta
\end{pmatrix}\ .
\end{equation*}
Assuming $H$ and $H_0$ are isospectral, one necessarily has $\mathrm{tr}(V)=0$ and therefore $\alpha=-\beta$. Furthermore, in case of isospectrality, one has $\det(H)=\det(H_0)$ and hence a simple calculation implies that $\alpha=c-a$ or $\alpha=0$. On the other hand, for any choice of $a \neq c \in \mathbb{R}$ together with $\alpha=c-a$ and $\beta=a-c$ one has
\begin{equation*}H=\begin{pmatrix} c & b \\ b & a
\end{pmatrix}\ ,
\end{equation*}
and this means that $H$ is isospectral to $H_0$ since they share the same characteristic polynomial. In other words, the potential $V$ cannot be recovered from knowing the spectrum of $H$ since we cannot distinguish the case
\begin{equation*}V=\begin{pmatrix} 0 & 0 \\ 0 & 0
\end{pmatrix}\ ,
\end{equation*}
from the case where
\begin{equation*}V=\begin{pmatrix} c-a & 0 \\ 0& a-c 
\end{pmatrix}\ .
\end{equation*}
Consequently, an Ambarzumian-type theorem does not hold for general symmetric $H_0 \in \mathbb{R}^{n \times n}$ and associated $H\in \mathbb{R}^{n \times n}$. However, for $n=2$ and $a=c$, one nevertheless has that isospectrality implies that $V=0$. This means that, for this special class of matrices, one does have an Ambarzumian-type theorem. At this point, it is instructive to remark that $a=c$ is naturally fulfilled whenever $H_0$ is an adjacency or Laplacian matrix of a path graph with two vertices. This already indicates why it is natural to study Ambarzumian-type theorems for the discrete Laplacian or the adjacency matrix over discrete graphs. 

\section{Main results I}
 Let $H_0 \in \mathbb{C}^{n \times n}$ be Hermitian for $n \geq 2$, $V=\mathrm{diag}(v_j) \in \mathbb{R}^{n \times n}$ a (real) diagonal matrix and set $H:=H_0+V$. Let $\lambda_j(H_0)$ and $\lambda_j(H)$ with $j=0,1,...,n-1$ denote the corresponding eigenvalues with normalized eigenvectors $g_j(H_0), g_j(H) \in \mathbb{C}^n$. We now state our first main result which forms a finite-dimensional generalization of \cite[Theorem~3.1]{DaviesAmba}. 
\begin{theorem}[Davies-type theorem for Hermitian matrices]\label{MainResultAmbaI} Let $H_0 \in \mathbb{C}^{n \times n}$ and $H \in \mathbb{C}^{n \times n}$ be defined as above. Furthermore, assume that 
\begin{equation}\label{ConditionAbleitung}
   \langle g_0(H_0),Vg_0(H_0)\rangle_{\mathbb{C}^n} \leq 0\ .
\end{equation}
Then, 
\begin{enumerate}[(i)]
\item $\lambda_0(H)\le \lambda_0(H_0)$, 
\item $\lambda_0(H)= \lambda_0(H_0)$ implies that $V|_{\supp(g_0(H_0))}=0$.
\end{enumerate}
\end{theorem}
\begin{proof} The proof employs the minimax-principle. Using this and \eqref{ConditionAbleitung} one gets
\begin{equation*}\begin{split}
    \lambda_0(H) \leq \langle g_0(H_0),Hg_0(H_0)\rangle_{\mathbb{C}^n}&=\lambda_0(H_0)+\langle g_0(H_0),Vg_0(H_0)\rangle_{\mathbb{C}^n} \\
    &\leq  \lambda_0(H_0)\ ,
    \end{split}
\end{equation*}
and this proves $(i)$. Furthermore, under the assumptions of $(ii)$ the above inequalities become equalities and hence $\langle g_0(H_0),Vg_0(H_0)\rangle_{\mathbb{C}^n}=0$ and $g_0(H_0)$ is an eigenvector of $H$ associated to $\lambda_0(H)$. This establishes $V g_0(H_0)=0$ and hence $(ii)$. 
\end{proof}
\begin{example}
    Note that, under the assumptions of Theorem~\ref{MainResultAmbaI}, $\lambda_0(H)=\lambda_0(H_0)$ does not necessarily imply $V=0$. For example, consider
    \begin{gather*}
        H_0 = \begin{pmatrix} 1 & 0 & -1 \\ 0 & 1 & 0 \\ -1 & 0 & 1 \end{pmatrix}\ .
    \end{gather*}
    Then $g_0(H_0)= \frac{1}{\sqrt{2}}(1, 0, 1)^T$ is a ground state to $\lambda_0(H_0)=0$. For
    \begin{gather*}
        V= \begin{pmatrix} 0 & 0 & 0 \\ 0 & 1 & 0 \\ 0 & 0 & 0 \end{pmatrix},
    \end{gather*}
one also readily has $\langle g_0(H_0),Vg_0(H_0)\rangle_{\mathbb{C}^3}=0$ as well as  $\lambda_0(H_0) = \lambda_0(H_0 + V)=0$. 
\end{example}
Theorem~\ref{MainResultAmbaI} now implies the following Ambarzumian-type theorem; here, $\boldsymbol{1} \in \mathbb{C}^n$ refers to the vector with all components equal to one. 
\begin{theorem}[Ambarzumian-type theorem I: Hermitian matrices]\label{MainResultAmbaII} Let $H_0 \in \mathbb{C}^{n \times n}$ and $H \in \mathbb{C}^{n \times n}$ be defined as above. Suppose that $\boldsymbol{1} \in \mathbb{C}^n$ is an eigenvector associated to the lowest eigenvalue of $H_0$. Then, $H_0$ and $H$ are isospectral if and only if $V=0$. 
\end{theorem}
\begin{proof}
    One direction is immediate. On the other hand, if $H_0$ and $H$ are isospectral, then $\mathrm{tr}(H_0) = \mathrm{tr}(H)$ and hence $\langle \boldsymbol{1}, V \boldsymbol{1}\rangle_{\mathbb{C}^n} = 0$. Consequently, by Theorem~\ref{MainResultAmbaI} we conclude $V=0$.
\end{proof}

\subsection{Application I: The discrete Laplacian on finite graphs}\label{SubsectionLaplacian}

Following \cite{KellerLenzGraphs}, let $X$ be a non-empty and finite set and $b: X \times X \rightarrow [0,\infty)$ a symmetric map satisfying
$b(x,x) = 0$ (meaning one has no loops); furthermore,
let $V: X \rightarrow \mathbb{R}$ be the external potential. Note that, with an abuse of notation, we sometimes also write the same potential as a diagonal matrix $V \in \mathbb{R}^{|X| \times |X|}$.

The discrete Laplacian is then the self-adjoint operator $ \mathcal{L}_{b,V}$ acting as
\[
\big(\mathcal{L}_{b,V}f \big)(x) := \sum_{y \in X} b(x,y) (f(x)-f(y)) + V(x)f(x)\ , \qquad  \text{$x \in X$}\ ,
\]
with associated quadratic form $\mathcal{Q}_{b,V}: \mathcal{D}_{b,V} \times \mathcal{D}_{b,V} \rightarrow \mathbb{C}$, $\mathcal{D}_{b,V}:=\mathbb{C}^{|X|}$, 
\[
\mathcal{Q}_{b,V}(f,g) := \frac{1}{2}\sum_{x,y \in X} b(x,y) \overline{(f(x)-f(y))} (g(x)-g(y)) + \sum_{x \in X}V(x)\overline{f(x)}g(x)\ , \quad f,g \in \mathcal{D}_{b,V}\ .
\]
\begin{remark} Assuming that $b(x,y)=1$ if $x\in X$ and $y \in X$ are connected and $b(x,y)=0$ otherwise, the Laplacian $\mathcal{L}_{b,V=0}$ can be written in the standard form $D-A$, where $A \in \mathbb{R}^{|X| \times |X|}$ is unweighted adjacency matrix and $D \in \mathbb{R}^{|X| \times |X|}$ the diagonal degree matrix. 
\end{remark}
One easily checks that $\boldsymbol{1} \in \mathbb{C}^n$ is an eigenvector of $\mathcal L_{b,V=0}$ and that the matrix representation of $\mathcal L_{b,V=0}$ satisfies the assumptions in Theorem~\ref{MainResultAmbaII}. We therefore obtain the following statement, versions of which have been stated before, see for instance \cite[Example~3.4]{DaviesAmba}.
\begin{theorem}[Ambarzumian-type theorem for the discrete Laplacian]\label{CorollaryLaplacian} Let $\mathcal{L}_{b,V}$ be the discrete Laplacian over a finite set $X$. Then, the self-adjoint operators $\mathcal{L}_{b,V}$ and $\mathcal{L}_{b,V=0}$ are isospectral if and only if $V=0$. 
\end{theorem}
With respect to Theorem~\ref{CorollaryLaplacian} there is an interesting observation: consider a path graph $X=\{0,...,n\}$ with $b(x,y)=1$ if and only if $|x-y|=1$ and $b(x,y)=0$ otherwise. Fix a parameter $h > 0$ and a map $v:X \rightarrow \mathbb{R}$ with $v(0)=\alpha h$, $v(n)=\beta h$ for $\alpha,\beta \in \mathbb{R}$ and $v(j)=0$ otherwise. Then, a discretization of the continuous Robin Laplacian over an interval with Robin parameters $\alpha,\beta \in \mathbb{R}$ (where $\alpha$ refers to the left interval end and $\beta$ to the right interval end) is given by the matrix 
$\frac{1}{h^2}\mathcal{L}_{b,v}$. Theorem~\ref{CorollaryLaplacian} then implies that the operator 
\begin{equation}\label{RobinLaplacian}
    \frac{1}{h^2}\mathcal{L}_{b,v}+V=  \frac{1}{h^2}\mathcal{L}_{b,v+h^2V}
\end{equation}
with an arbitrary non-zero diagonal matrix $V \in \mathbb{R}^{|X| \times |X|}$ such that $v+h^2V \neq 0$ is not isospectral to $\frac{1}{h^2}\mathcal{L}_{b,v=0}$, which corresponds to a discretization of the Neumann Laplacian. This means that the associated phenomenon from the continuous case as described in the introduction does not carry over to the discrete case, at least in the setting described; more explicitly, the discretization of the Neumann Laplacian over an interval is not isospectral to the above discretization of the Robin Laplacian with an additional external potential $V$ whenever $v+h^2V \neq 0$.

\subsection{Application II: The adjacency matrix of a finite graph} 
Using the notation from Section~\ref{SubsectionLaplacian}, the (weighted) adjacency matrix $A_b \in \mathbb{C}^{|X| \times |X|}$ is the Hermitian matrix with $(A_b)_{x,y}=b(x,y)$ if and only if $b(x,y)\neq 0$ and zero otherwise; the unweighted adjacency matrix $A \in \mathbb{C}^{n \times n}$ can be constructed by setting $A_{x,y}=1$ if and only if $b(x,y)\neq 0$ and zero in all other cases. 

Regarding Ambarzumian-type theorems for the (unweighted) adjacency matrix, one may refer to \cite{AmbaDisc} where -- among others -- it was proved that the (unweighted) adjacency matrix on a path graph is not isospectral to the (unweighted) adjacency matrix of a path graph perturbed by an additional potential; the authors obtained this result by comparing the associated characteristic polynomials. Naturally, an immediate question is what happens in case of more general (finite) discrete graphs? In this case, it seems difficult to argue on the level of characteristic polynomials. Quite surprisingly, however, there exists a rather elementary argument as to why an Ambarzumian-type theorem always holds for the (weighted as well as unweighted) adjacency matrix. One should also stress that this argument does not use Theorem~\ref{MainResultAmbaI} or Theorem~\ref{MainResultAmbaII} and that versions of the argument appeared before in the literature, for example in the proof of \cite[Proposition~2.13]{SaburovaInfinite}.
\begin{theorem}[Ambarzumian-type theorem for matrices with vanishing diagonal] \label{MainResult3}
    Consider a Hermitian matrix $H_0\in \mathbb{C}^{n\times n}$ such that $(H_0)_{jj}=0$ for $j=1,\ldots, n$ and let $V\in \mathbb R^{n\times n}$ be a diagonal matrix. Then $H_0+V$ and $H_0$ are isospectral if and only if $V=0$.
\end{theorem}
\begin{proof} Isospectrality implies $\mathrm{tr}(H_0+V)=\mathrm{tr}(H_0)$ and hence $\mathrm{tr}(V)=0$. Furthermore, one has $\mathrm{tr}((H_0+V)^2)=\mathrm{tr}(H_0^2)$ and hence $$\mathrm{tr}(H_0V)+\mathrm{tr}(VH_0)+\mathrm{tr}(V^2)=0\ .$$
Since $V \in  \mathbb{R}^{n \times n}$ is diagonal, one infers that $\mathrm{tr}(H_0V)=\mathrm{tr}(VH_0)=0$ and hence $\mathrm{tr}(V^2)=0$. This implies the statement.
\end{proof}
In particular, we conclude the following statement as a corollary; for complementary results on infinite periodic graphs see \cite[Proposition~2.13]{SaburovaInfinite}.

\begin{theorem}[Ambarzumian-type theorem for the adjacency matrix]\label{MainResult4} Consider a finite graph with associated weighted and unweighted adjacency matrices $A_b,A \in \mathbb{C}^{|X| \times |X|}$ and let $V \in \mathbb{R}^{|X| \times |X|}$ be a diagonal matrix. Then, $A_b+V$ and $A_b$ or $A+V$ and $A$ are isospectral if and only if $V=0$.
\end{theorem}
Note that this statement typically fails if one allows for \textit{complex} potentials $V$; see the recent paper~\cite{Faust}.
\section{Main results II}\label{SectionMainResultsII}

\subsection{Another Ambarzumian-type theorem} Let $B \in \mathbb{C}^{n \times n}$ be a Hermitian matrix with eigenvalues $\lambda_j(B)$, $j=0,1,...,n-1$. Ky Fan's variational principle \cite{KyFan} asserts that
\begin{equation*}
\sum_{j=0}^{k-1}\lambda_j(B) = \min\left\{\mathrm{tr}(PB): P\ \text{is an orthogonal projection of rank k}\right\}\ .
\end{equation*}
Furthermore, a minimizer is given by the projection onto the span of $k$ first eigenvectors. From his proof it also follows that the minimizer is unique whenever $\lambda_{k-1}(B) < \lambda_k(B)$. 
\begin{theorem}\label{TheoremKyFan} Let $H_0 \in \mathbb{C}^{n \times n}$ and $H\in \mathbb{C}^{n \times n}$ be given as above. Assume that $f:\mathbb{R} \rightarrow \mathbb{R}$ is continuous such that $f|_{\sigma(H_0)}:\sigma(H_0) \rightarrow \mathbb{R}$ is non-decreasing and that $f(H_0) \in \mathbb{C}^{n \times n}$ has a constant diagonal. Then, the following holds:
\begin{enumerate}
 \item[]
    \item[(i)] Whenever $f|_{\sigma(H_0)}:\sigma(H_0) \rightarrow \mathbb{R}$ is strictly increasing, $H_0$ and $H$ are isospectral if and only if $V=0$. 
    \item[]
    \item[(ii)] Whenever the graph induced on $\{1,...,n\}$ with an edge $\{i,j\}$ whenever $f(H_0)_{ij} \neq 0$ is connected, $H_0$ and $H$ are isospectral if and only if $V=0$. 
\end{enumerate}

\end{theorem}
\begin{proof} In both cases, isospectrality is clear whenever $V=0$. 

Now, consider the case $(i)$ and assume isospectrality: First, one has $\mathrm{tr}(V)=0$. Furthermore, since $f(H_0)$ has constant diagonal, one also has $\mathrm{tr}(f(H_0)V)=0$. Let $\mu_1 < \mu_2 < ... < \mu_r$ denote the distinct eigenvalues of $H_0$ and $Q_m$ be the projection onto the subspace generated by all eigenvectors up to those with eigenvalue $\mu_m$. 

By isospectrality of $H_0$ and $H$ and Ky Fan's variational principle, we obtain
\begin{equation}\label{EquationProof}
    \mathrm{tr}(Q_mV)=\mathrm{tr}(Q_mH)-\mathrm{tr}(Q_mH_0) \geq 0\ .
\end{equation}
We then write 
\begin{equation}\label{Representation}
    f(H_0)=f(\mu_r)\mathrm{Id}-\sum_{m=1}^{r-1}a_mQ_m
\end{equation}
with $a_m:=f(\mu_{m+1})-f(\mu_{m})$; by assumption, $a_m > 0$. We obtain
\begin{equation*}\begin{split}
\sum_{m=1}^{r-1}a_m\mathrm{tr}(Q_mV)=\mathrm{tr}((f(\mu_r)\mathrm{Id}-f(H_0))V)=0\ ,
\end{split}
\end{equation*}
and hence $\mathrm{tr}(Q_mV)=0$ for all $m=1,...,r$. Consequently, by \eqref{EquationProof}, isospectrality and Ky Fan's variational principle, it follows that $Q_m$ is also a (unique) minimizer for $H$. This implies that the eigenspaces $\ker(H_0-\mu_m)$ and $\ker(H-\mu_m)$ agree and hence $V$ acts as the null-operator on each eigenspace. This implies $V=0$. 

Now, consider the case $(ii)$ and assume isospectrality: Whenever $a_m  > 0$, \eqref{EquationProof} in combination with Ky Fan's variational principle implies that $Q_m$ is the minimizer in each case and hence commutes with $H_0$ as well as $H$. Consequently, since $V=H-H_0$, it also commutes with $V$. Then, by looking at \eqref{Representation}, we conclude that $V$ also commutes with $f(H_0)$. The statement then follows from the identity
\begin{equation*}
    [V,f(H_0)]_{ij}=(v_i-v_j)f(H_0)_{ij}\ ,
    \end{equation*}
    since the potential must be constant as soon as the induced graph is connected and then zero by $\mathrm{tr}(V)=0$.
\end{proof}
\begin{remark}
    Theorem~\ref{TheoremKyFan} generalizes Theorem~\ref{MainResult3}: in this case choose $f:\mathbb{R} \rightarrow \mathbb{R}$ as $x \mapsto x$. In addition, it generalizes Theorem~\ref{MainResultAmbaII} whenever one assumes that $H_0$ has a non-degenerate ground state: in this case choose a continuous $f$ such that $f|_{\sigma(H_0)}=-\mathrm{\chi}_{\{\lambda_0(H_0)\}}$; then $f(H_0)$ has constant diagonal since the ground state eigenvector is the constant vector. 
\end{remark}
\begin{example} Consider the real Hermitian matrix
\begin{equation*}
    H_0=\frac{1}{50}\begin{pmatrix}
75 & 0 & -45 & 60 \\
0 & 150 & -80 & -60 \\
-45 & -80 & 123 & 36 \\
60 & -60 & 36 & 102
\end{pmatrix}\ .
\end{equation*}
One has $\sigma(H_0)=\{0,1,3,5\}$. The projection on the subspace generated by the first two eigenvectors is given by
\begin{equation*}
    Q_2=\frac{1}{2}\begin{pmatrix}
    1& 0 & \frac{3}{5} & -\frac{4}{5} \\
0 & 1 & \frac{4}{5} & \frac{3}{5} \\
\frac{3}{5} & \frac{4}{5} & 1 & 0 \\
-\frac{4}{5} & \frac{3}{5} & 0 & 1
\end{pmatrix}\ .
\end{equation*}
As a consequence, choosing a continuous $f:\mathbb{R} \rightarrow \mathbb{R}$ with $f|_{\sigma(H_0)}=-\chi_{\{0,1\}}$, Theorem~\ref{TheoremKyFan} is applicable and one concludes that $H_0+V$ is, for $V \neq 0$, not isospectral to $H_0$. 
\end{example}
\subsection{Genericity results} In this section we want to show that, generically, Ambarzumian-type theorems cannot expected to hold for real symmetric matrices. Indeed, already above we have concluded that for two-by-two-matrices such theorems can only hold for matrices with a constant diagonal. In order to generalize these observations, let us introduce the map $ p:\mathbb{R}^n \rightarrow \mathbb{R}^n$ via
\begin{equation*}
    p_k(v)=\mathrm{tr}((H_0+V))^k\ , \quad k=1,...,n\ ,
\end{equation*}
and where $V=\mathrm{diag}(v)$. First, one has the following classical statement.
\begin{lemma}\label{LemmaAmba} Two hermitian matrices $A,B \in \mathbb{C}^{n \times n}$ are isospectral if and only if $\mathrm{tr}(A^k)=\mathrm{tr}(B^k)$ for all $k=1,...,n$. 
\end{lemma}
\begin{proof} If both $A$ and $B$ are isospectral, the statement follows from the spectral theorem. On the other hand, assuming $\mathrm{tr}(A^k)=\mathrm{tr}(B^k)$ for all $k=1,...,n$, one employs Newton's identities to conclude that the characteristic polynomials of $A$ and $B$ coincide and hence the two matrices are isospectral. 
\end{proof}
\begin{remark}
    Let $H_0$ and $H$ be defined as above. Then they are isospectral if and only if $p(v)=p(0)$. Consequently, for $H_0$ an Ambarzumian-type statement holds if and only if $p^{-1}(p(0))=\{0\}$. 
\end{remark}
Using Lemma~\ref{LemmaAmba}, we can then establish the following result which provides Hermitian matrices for which Ambarzumian-type theorems do not hold. 
\begin{theorem}[Absence of Ambarzumian-type results]\label{MainTheoremGeneric} Consider a Hermitian matrix $H_0 \in \mathbb{C}^{n \times n}$ with $n \geq 2$. Assume that $Dp(0)\in \mathbb{R}^{n \times n}$ is invertible. Then there exists a diagonal matrix $V \in \mathbb{R}^{n \times n}$ with $V \neq 0$ such that $H_0$ and $H_0+V$ are isospectral.
\end{theorem}
\begin{proof} In a first step one shows that $\|p(v)\|_{\mathbb{R}^n} \rightarrow \infty$ for $\|v\|_{\mathbb{R}^n} \rightarrow \infty$: Regarding the second component of $p:\mathbb{R}^{n} \rightarrow \mathbb{R}^n$ one has $p_2(v)=\mathrm{tr}((H_0+V)^2)=\|H_0+V\|^2_F \geq (\|V\|_F-\|H_0\|_F)^2$ and hence $\|V\|^2_F=\sum_{j=1}^{n}v^2_j$ yields the claim. 

Now, assume to the contrary that $p^{-1}(p(0))=\{0\}$. Then, since we assume that $Dp(0) \in \mathbb{R}^{n \times n}$ is invertible, $p(0) \in \mathbb{R}^n$ is a regular value. On the other hand, since $p_2(v) \geq 0$ for any $v \in \mathbb{R}^n$, one has $p^{-1}(w)=\emptyset$ for any vector $w \in \mathbb{R}^n$ with $w_2 < 0$ and hence any such vector $w \in \mathbb{R}^n$ is then also a regular value. This, however, immediately yields a contradiction with Theorem~\ref{PreimageRegValue} since the number of elements in the preimages are not the same modulo two. 
\end{proof}
Since we can identify every real symmetric matrix $H_0 \in \mathbb{R}^{n \times n}$ as a unique element of $\mathbb{R}^{\frac{n^2+n}{2}}$, we arrive at the following statement.
\begin{cor}[Genericity statement]\label{CorollaryGeneric} Let $n \geq 2$. Under the identification $\text{Sym}_n(\mathbb{R}) \cong \mathbb{R}^{\frac{n^2+n}{2}}$, consider the subset of all symmetric matrices $A \subset \{H_0 \in \mathbb{R}^{n \times n}: H_0=H^T_0\} \subset \mathbb{R}^{\frac{n^2+n}{2}}$ such that for every $H_0 \in A $ there does not exist a diagonal matrix $V \in \mathbb{R}^{n \times n}$ with $V \neq 0$ such that $H_0$ and $H_0+V$ are isospectral. Then $A\subset \mathbb{R}^{\frac{n^2+n}{2}}$ has measure zero (with respect to Lebesgue measure). 
\end{cor}
\begin{proof} First, a direct calculation shows that $Dp(0)=(k(H^{k-1}_0)_{jj})_{k,j}$ and hence $\det Dp(0)$ is a polynomial in $\mathbb{R}^{\frac{n^2+n}{2}}$ variables; note that it is not the zero-polynomial as can be seen by choosing $H_0$ diagonal with distinct values. Furthermore, by Theorem~\ref{MainTheoremGeneric} 
\begin{equation*}
  A \subset \{H_0 \in \mathbb{R}^{n \times n}: H_0=H^T_0\ ,\ \det Dp(0)=0\}
\end{equation*}
and since the zero-set of a (non-zero) polynomial has measure zero, the statement follows. 
\end{proof}
\section{appendix}\label{Appendix}
We now recall a statement from differential geometry that is important for us in the proof of Theorem~\ref{MainTheoremGeneric}. Let $p:\mathbb{R}^n \rightarrow \mathbb{R}^n$ be a smooth map. Recall that a value $y \in \mathbb{R}^n$ is called \textit{regular} if $Dp(v)$ is invertible for every $v \in p^{-1}(y)$. By Sard's theorem, regular values are dense in $\mathbb{R}^n$.
\begin{theorem}\label{PreimageRegValue} Let $p:\mathbb{R}^n \rightarrow \mathbb{R}^n$ be a smooth map such that $\|p(v)\|_{\mathbb{R}^n} \rightarrow \infty$ for $\|v\|_{\mathbb{R}^n} \rightarrow \infty$. Then the following holds:
\begin{enumerate}
\item[]
    \item[(i)] For every regular value $y \in \mathbb{R}^n$, $p^{-1}(y)$ is a finite set. Furthermore, there exists a neighborhood of $y$ such that the cardinality of $p^{-1}(y)$ and $p^{-1}(y^{\prime})$ coincide for all $y^{\prime}$ in that neighborhood. 
    \item[]
     \item[(ii)] If $y,y^{\prime} \in \mathbb{R}^n$ are two different regular values, then the cardinality of $p^{-1}(y)$ and $p^{-1}(y^{\prime})$ coincide up to modulo 2. 
\end{enumerate}
\end{theorem}
\begin{proof} For completeness, we include a proof of this statement which essentially uses methods presented in \cite{Milnor}.

We start with $(i)$: First of all, the inverse function theorem asserts that the set $p^{-1}(y)$ is discrete, it is closed by continuity, and it is also bounded by the growth assumptions on $p$ and hence it is compact. This shows that it is actually finite. Assume now that $v_1,...,v_N$ are the points in $p^{-1}(y)$ with disjoint open neighborhoods $U_1,...,U_N$ which are diffeomorphic to open neighborhoods of $y$. Then $p(\mathbb{R}^n\setminus \bigcup_j U_j)$ is a closed set (by the growth assumption on $p$) that does not contain $y$. Hence there exists an open neighborhood of $y$ whose preimage is contained in $\bigcup_j U_j$ such that, for each $y^{\prime}$ in this neighborhood, one has exactly one preimage per $U_j$ by the inverse function theorem and this proves $(i)$.

In order to prove $(ii)$, consider the map $H:\mathbb{R}^n \times [0,1] \rightarrow \mathbb{R}^n$ defined via $H(v,t):=p(v)-t(y^{\prime}-y)$. Sard's theorem for manifolds with boundary asserts that we can find a $z \in \mathbb{R}^n$ arbitrarily close to $0$ such that $y+z$ is a regular value of $p$ and $H$ (on $\mathbb{R}^n \times (0,1)$), $y^{\prime}+z$ is a regular value of $p$ and $\# p^{-1}(y+z)=\# p^{-1}(y)$, $\# p^{-1}(y^{\prime}+z)=\# p^{-1}(y^{\prime})$ by $(i)$. Now, due to the growth assumption on $p$, one infers that the set $H^{-1}(y+z)$ is compact and, due to the preimage theorem, it is also a smooth one-dimensional manifold with boundary $p^{-1}(y+z) \times \{0\} \cup p^{-1}(y^{\prime}+z) \times \{1\}$. Due to the classification theorem of compact smooth one-dimensional manifolds, the boundary consists of an even number of points and hence $(ii)$ follows. 
\end{proof}

\section*{Author Declarations}
\subsection*{Conflict of Interest}
The authors have no conflicts to disclose.

\subsection*{Author Contributions}
All authors contributed equally to this work.

\subsection*{Data Availability}
Data sharing is not applicable to this article as no new data were created or analyzed.

\subsection*{Acknowledgement}{M.H. was supported by the Fundação para a Ciência e a Tecnologia (FCT), Portugal, within the scope of the
project SpectralOPs, reference 2023.13921.PEX. This version of the paper was developed based on a previous preprint of the authors in combination with AI tools. In particular, the results in Section~\ref{SectionMainResultsII} and Section~\ref{Appendix} have been developed with AI support. All results have been independently verified by the authors and they take full responsibility of the content of the paper.}

\newcommand{\etalchar}[1]{$^{#1}$}
\def\cprime{$'$} \def\polhk#1{\setbox0=\hbox{#1}{\ooalign{\hidewidth \lower1.5ex\hbox{`}\hidewidth\crcr\unhbox0}}} \def\cprime{$'$} \def\polhk#1{\setbox0=\hbox{#1}{\ooalign{\hidewidth \lower1.5ex\hbox{`}\hidewidth\crcr\unhbox0}}}


\begin{thebibliography}{HEF{\etalchar{+}}21}

\bibitem[Ale78]{Alexander}
J.~C. Alexander.
\newblock Matrices, eigenvalues, and complex projective space.
\newblock {\em Amer. Math. Monthly}, 85(9):727--733, 1978.

\bibitem[Amb29]{Ambarz}
V.~Ambarzumian.
\newblock Über eine {F}rage der {E}igenwerttheorie.
\newblock {\em Zeitschrift für Physik}, 53:690--695, 1929.

\bibitem[BK24]{BKAmba}
P.~Bifulco and J.~Kerner.
\newblock A note on {A}mbarzumian's theorem for quantum graphs.
\newblock {\em Arch. Math.}, 123(1):95--102, 2024.

\bibitem[BKR26]{BKR}
P.~Bifulco, J.~Kerner, and C.~Rose.
\newblock Spectral comparison results for {L}aplacians on discrete graphs.
\newblock {\em Arch. Math.}, 127:407--418, 2026.

\bibitem[BKS18]{BKS}
J.~Boman, P.~Kurasov, and R.~Suhr.
\newblock Schr\"{o}dinger operators on graphs and geometry {II}. {S}pectral estimates for {$L_1$}-potentials and an {A}mbartsumian theorem.
\newblock {\em Integral Equations Operator Theory}, 90(3):Paper No. 40, 24, 2018.

\bibitem[Bor46]{Borg}
G.~Borg.
\newblock {Eine Umkehrung der {S}turm-{L}iouvilleschen Eigenwertaufgabe: Bestimmung der Differentialgleichung durch die Eigenwerte}.
\newblock {\em Acta Mathematica}, 78:1 -- 96, 1946.

\bibitem[Bro25]{BrolinAmba}
A.~Brolin.
\newblock Ambarzumian theorem for quantum graphs with magnetic potential.
\newblock {\em Math. Scand.}, 131(2):391--398, 2025.

\bibitem[CFK26]{Faust}
J.~Cobb, M.~Faust, and A.~Kretschmer.
\newblock Inverse eigenvalue problems, {F}loquet isospectrality and the {H}ilbert--{C}how morphism.
\newblock Preprint, \href{https://arxiv.org/abs/2601.09604}{arXiv:2601.09604}, 2026.

\bibitem[Dav13]{DaviesAmba}
E.~B. Davies.
\newblock An inverse spectral theorem.
\newblock {\em J. Operator Theory}, 69(1):195--208, 2013.

\bibitem[DH56]{DowningHouseholder}
A.~C. Downing and A.~S. Householder.
\newblock Some inverse characteristic value problems.
\newblock {\em J. Assoc. Comput. Mach.}, 3:203--207, 1956.

\bibitem[Fan49]{KyFan}
K.~Fan.
\newblock On a theorem of {W}eyl concerning eigenvalues of linear transformations. {I}.
\newblock {\em Proc. Nat. Acad. Sci. U.S.A.}, 35:652--655, 1949.

\bibitem[Fri77]{Friedland}
S.~Friedland.
\newblock Inverse eigenvalue problems.
\newblock {\em Linear Algebra Appl.}, 17(1):15--51, 1977.

\bibitem[Had68]{Hadeler}
K.~P.. Hadeler.
\newblock Ein inverses {E}igenwertproblem.
\newblock {\em Linear Algebra Appl.}, 1(1):83--101, 1968.

\bibitem[Har87]{HarrellAmba}
E.~M. Harrell.
\newblock On the extension of {A}mbarzumian's inverse spectral theorem to compact symmetric spaces.
\newblock {\em American Journal of Mathematics}, 109(5):787--795, 1987.

\bibitem[HEF{\etalchar{+}}21]{AmbaDisc}
B.~Hatino\u{g}lu, J.~Eakins, W.~Frendreiss, L.~Lamb, S.~Manage, and A.~Puente.
\newblock Ambarzumian-type problems for discrete {S}chr\"{o}dinger operators.
\newblock {\em Complex Anal. Oper. Theory}, 15(8):Paper No. 118, 13, 2021.

\bibitem[Hoc73]{HochSL}
H.~Hochstadt.
\newblock The inverse {S}turm-{L}iouville problem.
\newblock {\em Comm. Pure Appl. Math.}, 26:715--729, 1973.
\newblock Collection of articles dedicated to Wilhelm Magnus.

\bibitem[Kac66]{MKac}
M.~Kac.
\newblock Can one hear the shape of a drum?
\newblock {\em Amer. Math. Monthly}, 73(4, part II):1--23, 1966.

\bibitem[Kis23]{KissAmba}
M.~Kiss.
\newblock An {A}mbarzumian-type theorem on graphs with odd cycles.
\newblock {\em Ukrainian Mathematical Journal}, 74:1916–1923, 2023.

\bibitem[KLW21]{KellerLenzGraphs}
M.~Keller, D.~Lenz, and R.~K. Wojciechowski.
\newblock {\em Graphs and {D}iscrete {D}irichlet {S}paces}.
\newblock Springer Nature Switzerland AG, 2021.

\bibitem[Kur19]{PKUnderstanding}
P.~Kurasov.
\newblock Understanding quantum graphs.
\newblock {\em Acta Physica Polonica A}, 136(5), 2019.

\bibitem[Lev49]{LevinsomInverse}
N.~Levinson.
\newblock The inverse {S}turm-{L}iouville problem.
\newblock {\em Mat. Tidsskr. B}, 1949:25--30, 1949.

\bibitem[LG64]{LevGasAmba}
B.~M. Levitan and M.~G. Gasymov.
\newblock Determination of a differential equation by two spectra.
\newblock {\em Uspehi Mat. Nauk}, 19(2(116)):3--63, 1964.

\bibitem[Mil65]{Milnor}
J.~W. Milnor.
\newblock {\em Topology from the differentiable viewpoint}.
\newblock University Press of Virginia, Charlottesville, VA, 1965.
\newblock Based on notes by David W. Weaver.

\bibitem[MS22]{SMMatrix}
T.~Maci\c{a}\.{z}ek and U.~Smilansky.
\newblock Can one hear a matrix? {R}ecovering a real symmetric matrix from its spectral data.
\newblock {\em Ann. Henri Poincar\'{e}}, 23(5):1791--1829, 2022.

\bibitem[Piv05]{PivoAmba}
V.~Pivovarchik.
\newblock Ambarzumian’s theorem for a {S}turm-{L}iouville boundary value problem on a star-shaped graph.
\newblock {\em Funct Anal Its Appl}, 39:148–151, 2005.

\bibitem[Sab26]{SaburovaInfinite}
N.~Saburova.
\newblock Spectral invariants for discrete {S}chr\"{o}dinger operators on periodic graphs.
\newblock {\em J. Spectr. Theory}, 16(2):505--539, 2026.

\end{thebibliography}
\end{document}